\documentclass[eqno] {amsart}
\usepackage{amscd}
\usepackage{amssymb}
\usepackage{cite}
\usepackage{enumerate}
\usepackage{amsmath}
\usepackage[usenames, dvipsnames]{color}
\usepackage{bookmark}
\bookmarksetup{level=section}

\def\w*lim{\mathop{\mbox{\textup{w*-lim}}}}

\usepackage{mathrsfs}

\newtheorem{theorem}{\sc \textbf{Theorem}}[section]
\newtheorem{lemma}[theorem]{\sc \textbf{Lemma}}

\newcounter{cnt1}
\newcounter{cnt2}
\newcounter{cnt3}
\newcounter{cnt4}
\newcommand{\blr}{\begin{list}{$($\roman{cnt1}$)$} {\usecounter{cnt1}
 \setlength{\topsep}{0pt} \setlength{\itemsep}{0pt}}}
\newcommand{\blR}{\begin{list}{\Roman{cnt4}.\ } {\usecounter{cnt4}
 \setlength{\topsep}{0pt} \setlength{\itemsep}{0pt}}}
\newcommand{\bla}{\begin{list}{$($\alph{cnt2}$)$} {\usecounter{cnt2}
 \setlength{\topsep}{0pt} \setlength{\itemsep}{0pt}}}
\newcommand{\bln}{\begin{list}{$($\arabic{cnt3}$)$} {\usecounter{cnt3}
 \setlength{\topsep}{0pt} \setlength{\itemsep}{0pt}}}
\newcommand{\el}{\end{list}}

\usepackage{hyperref}					% Reference

\newcommand{\cA}{{\mathcal A}}
\newcommand{\cB}{{\mathcal B}}

\newcommand{\cF}{{\mathcal F}}

\newcommand{\cH}{{\mathcal H}}

\newcommand{\cM}{{\mathcal M}}

\newcommand{\cP}{{\mathcal P}}

\newcommand{\cR}{{\mathcal R}}

\usepackage[english]{babel}
\usepackage{yfonts}

\begin{document}

\title[The algebra of thin measurable operators is directly finite]{The algebra of thin measurable operators is directly finite}

\author[A. Bikchentaev]{A. Bikchentaev}
\address[Airat M. Bikchentaev]{Kazan Federal University, 18 Kremlyovskaya str., Kazan, 420008 Russia \emph{E-mail~:} {\tt Airat.Bikchentaev@kpfu.ru}
}

\begin{abstract}
Let $\mathcal{M}$ be a semifinite von Neumann algebra on a Hilbert space $\cH$ equipped with a faithful normal semifinite trace $\tau$,  $S(\mathcal{M},\tau)$ be the ${}^*$-algebra of all $\tau$-measurable operators. Let $S_0(\mathcal{M},\tau)$ be the ${}^*$-algebra of all $\tau$-compact operators and
$T(\mathcal{M},\tau)=S_0(\mathcal{M},\tau)+\mathbb{C}I$ be the  ${}^*$-algebra of all operators $X=A+\lambda I$
with $A\in S_0(\mathcal{M},\tau)$ and $\lambda \in \mathbb{C}$. We prove that every operator of $T(\mathcal{M},\tau)$ that is left-invertible in  $T(\mathcal{M},\tau)$ is in fact invertible in  $T(\mathcal{M},\tau)$.
It is a generalization of   Sterling Berberian theorem (1982) on the subalgebra of thin  operators in $\cB (\cH)$.
For the singular value function $\mu(t; Q)$ of  $Q=Q^2\in S(\mathcal{M},\tau)$  we have  $\mu(t; Q)\in \{0\}\bigcup
[1, +\infty)$ for all $t>0$. It  gives the positive answer to the question posed by Daniyar Mushtari in  2010.
\end{abstract}

\subjclass[2010]{16E50, 46L51. \hfill Version~: \today.}

\maketitle

\section{Introduction}

In this paper we extend the Sterling Berberian's result \cite{Ber82} (see also \cite{Hal81}) on direct finiteness of the algebra of thin operators on a infinite-dimensional Hilbert space to the Irving Segal's  non-commutative integration setting \cite{Seg53}.
Let $\mathcal{M}$ be a semifinite von Neumann algebra on a Hilbert space $\cH$ equipped with a faithful normal semifinite trace $\tau$,  $S(\mathcal{M},\tau)$ be the ${}^*$-algebra of all $\tau$-measurable operators. Let $S_0(\mathcal{M},\tau)$ be the ${}^*$-algebra of all $\tau$-compact operators and
$T(\mathcal{M},\tau)=S_0(\mathcal{M},\tau)+\mathbb{C}I$ be the  ${}^*$-algebra of all operators $X=A+\lambda I$
with $A\in S_0(\mathcal{M},\tau)$ and a complex num\-ber~$\lambda$. We prove that every operator of $T(\mathcal{M},\tau)$   left-invertible in  $T(\mathcal{M},\tau)$ is actually invertible in  $T(\mathcal{M},\tau)$ (Theorem \ref{t3.6}).
For the singular value function $\mu(t; Q)$ of  $Q=Q^2\in S(\mathcal{M},\tau)$  we have  $\mu(t; Q)\in \{0\}\bigcup
[1, +\infty)$ for all $t>0$ (Theorem \ref{t3.7}). It  is the positive answer to the question by Daniyar Mushtari of year 2010.

The author sincerely thank Vladimir Chilin for useful discussions of the results  presented in this paper. 

\section{Preliminaries}\label{prel}

Let ${\mathcal M}$ be a von Neumann algebra of operators on a Hilbert space
 $\mathcal H$, let $\cP(\mathcal{M} )$ be the lattice of projections in ${\mathcal M}$,  $I$  be the unit of
$\mathcal{M}$.
Also $\mathcal{M}^+$  denotes  the cone of positive elements in $\mathcal{M}$.
A mapping  $\varphi
:\mathcal{M}^+\to [0, +\infty]$ is called {\it a trace}, if
$\varphi (X+Y)=\varphi (X)+ \varphi (Y)$, $\varphi (\lambda
X)=\lambda \varphi (X)$ for all  $X, Y \in \mathcal{M}^+$, $\lambda
\ge 0$ (moreover, $0 \cdot(+\infty)\equiv 0$);
$\varphi (Z^*Z)=\varphi (ZZ^*)$ for all  $Z \in \mathcal{M}$.
A trace  $\varphi$ is called
 {\it faithful}, if $\varphi (X)>0$ for all  $X \in
\mathcal{M}^+$, $X \not=0$;
 {\it normal}, if $X_i \uparrow X $ $(X_i,X \in \mathcal{M}^+
)\Rightarrow \varphi (X)=\sup \varphi (X_i)$;
{\it semifinite}, if $ \varphi (X)=\sup \{
\varphi (Y): \; Y\in \mathcal{M}^+, \; Y \le X, \; \varphi
(Y)<+\infty\}$ for every $X \in \mathcal{M}^+$.

An operator   on $\mathcal{H}$ (not necessarily bounded or densely defined) 
is said to be  {\it affiliated to the von Neumann algebra} $\mathcal{M}$ if it commutes with any
unitary operator from the commutant   $\mathcal{M}'$ of the algebra $\mathcal{M}$. 
Let   $\tau$ be a faithful normal semifinite trace  on
$\mathcal{M}$. A closed operator $X$, affiliated to  
$\mathcal{M}$ and possesing a domain  $\mathfrak{D}(X)$
everywhere dense
 in $\mathcal{H}$ is said to be {\it $\tau$-measurable} if,
   for any $\varepsilon >0$, there exists a  $P\in
\cP(\mathcal{M})$ such that $ P\mathcal{H}\subset
\mathfrak{D}(X)$  and $ \tau (I-P)<\varepsilon $. The set 
$S(\mathcal{M}, \tau )$ of all $\tau$-measurable operators is a
${}^*$-algebra under  passage to the adjoint operator,
multiplication by a scalar, and operations of strong addition and 
multiplication resulting from the closure of the ordinary operations \cite{Seg53}, \cite{Nel74}.
Let $\mathcal{L}^+$ and $\mathcal{L}^{\text{\rm h}}$ denote the positive and Hermitian parts of  a family
$\mathcal{L}\subset S(\mathcal{M},\tau)$, respectively.
We denote  by  $\leq$ the partial order in
$S(\cM,\tau)^{\text{\rm h}}$ generated by its proper cone $S(\cM,\tau)^+$.
If $X\in S(\cM,\tau)$,  then $|X|= \sqrt{X^*X}\in S(\cM,\tau)^+$.
The generalized singular value function $\mu(X):t\rightarrow \mu(t;X)$ of
the operator $X$ is defined by setting
$$
\mu(s;X)
=
\inf\{\|XP\|: \; P \in \cP (\mathcal{M}) \mbox{  and } \tau(I-P)\leq s\}.
$$

\begin{lemma}\label{l2.1} {\rm (see \cite{FK86})}
{\it We have 
$\mu (s+t;XY)\leq \mu (s; X)\mu (t;Y)$  for all  $X, Y \in S(\mathcal{M}, \tau )$ and $s, t>0$. }
\end{lemma}

The sets
$U({\varepsilon, \delta})=\{X \in S(\mathcal{M}, \tau ):
 \; (\|XP\|\le \varepsilon \;
{\text{\rm and}} \;  \tau(I-P)\leq \delta  \; {\text{\rm for some}} \; P \in \cP (\mathcal{M})) \}$,
where $ \varepsilon >0,\, \delta >0$, form a base at $0$ for a metrizable vector topology $t_{\tau}$
on  $S(\mathcal{M}, \tau )$, called {\it the measure topology} \cite{Nel74}.
Equipped with this topology, 
$S(\mathcal{M}, \tau )$ is a  complete metrizable topological ${}^*$-algebra in which $\mathcal{M}$ is dense.
We will write $X_n \buildrel{\tau}\over
\longrightarrow X$ if a sequence $\{X_n{\}}_{n=1}^\infty$ converges to  $X \in S(\mathcal{M}, \tau )$
in the measure topology on  $S(\mathcal{M}, \tau )$. 

The set of  $\tau$-compact operators 
$S_0(\mathcal{M}, \tau )=
\{X \in  S(\mathcal{M}, \tau ) :  \lim \limits_{t\to \infty }\mu_t(X)=0\}$
is an ideal in  $S(\mathcal{M}, \tau )$. 
For any closed and densely defined linear operator $X:\mathfrak{D}\left( X\right) \rightarrow \cH $,
the \emph{null projection} ${\rm n}(X)={\rm n}(|X|)$ is the projection onto its kernel $\mbox{Ker} (X)$,
 the \emph{range projection } ${\rm r}(X)$ is the projection onto the closure of its range $\mbox{Ran}(X)$ and the \emph{support projection} ${\rm supp}(X)$ of $X$ is defined by ${\rm supp}(X) =I - n(X)$.

The two-sided ideal $\cF (\cM, \tau)$ in $\cM$ consisting of all elements of $\tau$-finite range is defined by
$$\cF(\cM, \tau)=\{X\in \cM ~:~ \tau({\rm r}(X)) <\infty\} = \{X \in \cM ~:~ \tau({\rm supp}(X)) <\infty\}.$$
Equivalently, $\cF(\cM, \tau)=\{X\in \cM : \mu(t;X)=0 \; \mbox{for some }t>0\}$.
Clearly, $S_0(\cM,\tau)$ is the closure of $\cF(\cM, \tau)$ with respect to the measure topology \cite{DP14}.

\section{The main results}

Throughout the sequel, let $\mathcal M$ be an arbitrary semifinite von Neumann algebra, with some
distinguished faithful normal semifinite trace $\tau$.

\begin{lemma}\label{l3.1} 
We have  $|X|\in T(\mathcal{M}, \tau)$ for every $X\in T(\mathcal{M}, \tau)$.
\end{lemma}
\begin{proof}
The ideal $\cF (\cM,\tau)$ is a $C^*$-subalgebra in $\cM$. Hence  $F(\mathcal{M}, \tau)=\cF (\cM,\tau)+\mathbb{C}I$ is a unital  $C^*$-subalgebra in $\cM$ and if $X\in F(\mathcal{M}, \tau)$, then $|X|\in F(\mathcal{M}, \tau)$.
Assume that $X\in T(\mathcal{M}, \tau)$, i.e., $X=A+\lambda I$ with $A\in S_0(\mathcal{M}, \tau)$ and
$\lambda \in \mathbb{C}$. Since $\cF (\cM,\tau) $ is $t_{\tau}$-dense in  $S_0(\mathcal{M}, \tau)$, there exists a sequence  $\{A_n\}_{n=1}^{\infty} \subset \cF (\cM,\tau) $ such that $A_n \buildrel{\tau}\over \longrightarrow A$
as $n \to \infty$. Then the sequence  $X_n=A_n+\lambda I$, $n\in \mathbb{N}$, lies in  $F(\mathcal{M}, \tau)$
and  $t_{\tau}$-converges to the operator $X$ as  $n \to \infty$. According to the results given above, $|X_n|=B_n+| \lambda | I$
with  some  $B_n\in F(\mathcal{M}, \tau)^{\text{\rm h}}$, $n\in \mathbb{N}$.
Since  $X_n \buildrel{\tau}\over \longrightarrow X$
as $n \to \infty$, we have $X_n^* \buildrel{\tau}\over \longrightarrow X^*$ as $n \to \infty$ by  $t_{\tau}$-continuity
of the involution in $S(\mathcal{M}, \tau)$. Then via joint  $t_{\tau}$-continuity
of the multiplication in $S(\mathcal{M}, \tau)$ we have $X_n^* X_n\buildrel{\tau}\over \longrightarrow X^*X$ as $n \to \infty$. Therefore we obtain  $|X_n|  \buildrel{\tau}\over \longrightarrow |X|$
as $n \to \infty$ by  $t_{\tau}$-continuity of the real function $f(t)=\sqrt{t}$, $t\geq 0$ \cite{Tik87}. Thus
the sequence  $\{B_n\}_{n=1}^{\infty}$  $t_{\tau}$-converges to a some operator $B\in S_0(\mathcal{M}, \tau)^{\text{\rm h}}$ and $|X|=B+| \lambda | I$.
\end{proof}

\begin{lemma}\label{l3.2} 
{\rm (see \cite[Corollary 2.4]{Bik14})}
If   $X\in T(\mathcal{M}, \tau)$ and $XX^*\leq X^*X$, then  $XX^*=X^*X$.
\end{lemma}

\begin{lemma}\label{l3.3} 
The idempotents of    $T(\mathcal{M}, \tau)$ are the operators  $P$,  $I-P$, where $P$ runs over the  idempotent
operators of $S_0(\mathcal{M}, \tau)$.
\end{lemma}
\begin{proof} 
Assume that $X=A+\lambda I \in T(\mathcal{M}, \tau)$ and
$X^2=X$. Then  $A^2+2\lambda A+\lambda^2 I=A+\lambda I$, i.e., $\lambda \in \{0, 1\}$. If  $\lambda =0$, then
$A^2=A$ and $A\in S_0(\mathcal{M}, \tau)$ is an  idempotent operator.
Then $I -A\in T(\mathcal{M}, \tau)$ and is  also an idempotent. If  $\lambda =1$, then
$A^2=-A=(-A)^2$ and $-A\in S_0(\mathcal{M}, \tau)$ is an  idempotent operator.
Then $I -(-A)\in T(\mathcal{M}, \tau)$ and is also an idempotent.
\end{proof}

Consider $F_0(\mathcal{M}, \tau)=\{A\in S_0(\mathcal{M}, \tau):\; \tau ({\rm r}(A))<+\infty\}$ and 
$\cA (\mathcal{M}, \tau)=F_0(\mathcal{M}, \tau)+\mathbb{C}I$. Then $\cA (\mathcal{M}, \tau)$ is a
${}^*$-subalgebra of  $T(\mathcal{M}, \tau)$.

\begin{lemma}\label{l3.4} 
$\cA (\mathcal{M}, \tau)$ contains every idempotent of    $T(\mathcal{M}, \tau)$.
\end{lemma}
\begin{proof} 
Let $Q$ be an   idempotent operator of    $S(\mathcal{M}, \tau)$. Then
$$
(Q+Q^*-I)^2=I+(Q-Q^*)(Q-Q^*)^*
$$
and by \cite[Theorem 2.21]{Bik16} there exists a unique ``range'' projection $Q^{\sharp}\in \cP (\cM )$,  defined by the formula $Q^{\sharp}=Q(Q+Q^*-I)^{-1}$ with $(Q+Q^*-I)^{-1}\in \cM$ and subject to the condition $Q^{\sharp} \cdot S(\cM , \tau )=Q \cdot S(\cM , \tau )$. By \cite[Theorem 2.23]{Bik16}  there exists a unique decomposition $Q=P+Z$, where $P=Q^{\sharp}\in \cP (\cM )$
and $Z\in S(\mathcal{M}, \tau)$ is a nilpotent  so that $Z^2=0$ and $ZP=0$, $PZ=Z$.
Thus $QP=P$ and $PQ=Q$. Assume that $Q\in S_0(\mathcal{M}, \tau)$. Since $QP=P$, we have $P\in S_0(\mathcal{M}, \tau)$. Since the singular function $\mu (t; P)=\chi_{(0, \tau (P)]} (t)$ for all $t>0$, we conclude that
$P\in \cF (\mathcal{M}, \tau)$. Then by equality $PQ=Q$, we have $Q\in F_0(\mathcal{M}, \tau)$ and
apply  Lemma \ref{l3.3}. 
\end{proof}

\begin{lemma}\label{l3.5} 
 $F_0(\mathcal{M}, \tau)$ is a regular ring.
\end{lemma}
\begin{proof} 
We show that for every operator  $A\in F_0(\mathcal{M}, \tau)$ the equation $AXA=A$ possesses a solution in  $F_0(\mathcal{M}, \tau)$. For  $A\in F_0(\mathcal{M}, \tau)$ the range projection ${\rm r}(A)$ and the support
 projection ${\rm supp} (A)$ lie in 
$\cF(\mathcal{M}, \tau)$. Consider the projection $P={\rm r}(A)\bigvee {\rm supp}(A)$ in $\cF(\mathcal{M}, \tau)$ and
the reduced von Neumann algebra $\cM_P=P\cM P$, the reduced  faithful normal finite trace $\tau_P$ with
$\tau_P(X)=\tau (PXP)$, $X\in \cM^+_P$. The algebra $\cM_P$ is finite, therefore
$S(\mathcal{M}_P, \tau_P)$ is a regular ring by \cite[Theorem 4.3]{Sai71}. Since $A\in S(\mathcal{M}_P, \tau_P)$,
 the equation $AXA=A$ admits a solution in  $S(\mathcal{M}_P, \tau_P)\subset F_0(\mathcal{M}, \tau)$.
\end{proof}

Idempotents $P, Q$ of a ring $\cR$ are said to be {\it equivalent} (in $\cR$), written $P\sim Q$, if there exist elements
$X, Y \in \cR$ such that $XY=P$ and $YX=Q$ (replacing $X, Y$ by $PXQ$, $QYP$, one can suppose that $X\in P\cR Q$, $Y\in Q\cR P$ \cite[p. 22]{Kap68}. Projections (=self-adjoint idempotents) $P, Q$ of a ring with involutions 
 are said to be {\it ${}^*$-equivalent}  if there exists an  element $X$ such that $XX^*=P$ and $X^*X=Q$.

\begin{theorem}\label{t3.6}
If   $X, Y \in T(\mathcal{M}, \tau)$ such that $XY=I$, then $YX=I$.
\end{theorem}
\begin{proof}
In the terms of ring theory, we assert that the ring $T(\mathcal{M}, \tau)$ is ``directly finite''
\cite[p. 49]{Goo79}. Since $F_0(\mathcal{M}, \tau)$ (by Lemma \ref{l3.5}) and $\cA(\mathcal{M}, \tau)/F_0(\mathcal{M}, \tau)\cong \mathbb{C}$ are both  regular rings,  $\cA(\mathcal{M}, \tau)$ is a regular ring \cite[p. 2, Lemma 1.3]{Goo79}; since, moreover, the involution of  $\cA(\mathcal{M}, \tau)$ is proper ($AA^*=0$ implies $A=0$), the algebra $\cA(\mathcal{M}, \tau)$ is
${}^*$-regular in the sense of von Neumann \cite[p. 229]{Ber72}.

If $X, Y$ are elements of $T(\mathcal{M}, \tau)$ such that $XY=I$, then $P=YX$ is an idempotent of  $T(\mathcal{M}, \tau)$ such that $P\sim I$ in  $T(\mathcal{M}, \tau)$. By Lemma \ref{l3.4} we have $P\in  \cA(\mathcal{M}, \tau)$;
since $ \cA(\mathcal{M}, \tau)$ is ${}^*$-regular, there exists a projection $Q\in  \cA(\mathcal{M}, \tau)$
such that $Q\cdot  \cA(\mathcal{M}, \tau)=P\cdot  \cA(\mathcal{M}, \tau)$ \cite[p. 229, Proposition 3]{Ber72}.
Then $P\sim Q$ in  $\cA(\mathcal{M}, \tau)$  \cite[p. 21, Theorem 14]{Kap68}, a fortiori $P\sim Q$ in  $T(\mathcal{M}, \tau)$; already $P\sim I$ in  $T(\mathcal{M}, \tau)$, so $Q\sim I$ in  $T(\mathcal{M}, \tau)$ by transitivity.
Since  $T(\mathcal{M}, \tau)$ satisfies the ``square root'' axiom (SR) and contains square roots of its positive elements (see
Lemma \ref{l3.1} and \cite[p. 90]{Kap68}), it follows that the projections $P, I$ are ${}^*$-equivalent in 
$T(\mathcal{M}, \tau)$  \cite[p. 35, Theorem 27]{Kap68}, say $X\in T(\mathcal{M}, \tau)$ with $XX^*=P$, $X^*X=I$.
By Lemma \ref{l3.2}, $P=I$; then $Q\cdot  \cA(\mathcal{M}, \tau)=P\cdot  \cA(\mathcal{M}, \tau)= \cA(\mathcal{M}, \tau)$ shows that $P=I$, that is, $YX=I$.
\end{proof}

Theorem \ref{t3.6} can obviously be reformulated as follows: if  $A, B \in S_0(\mathcal{M}, \tau)$ and
$A+B+AB=0$, then $AB=BA$. On invertibility in $S(\mathcal{M}, \tau)$ see \cite{Tem10}, \cite{Bik21} and
 \cite{Bik22}.

\begin{theorem}\label{t3.7}
If   $Q\in S(\mathcal{M}, \tau)$ is such that $Q^2=Q$, then $\mu (t; Q)\in \{0\}\bigcup [1, +\infty)$ for all $t>0$.
For the symmetry $U=2Q-I$ we have $\mu (t; U) \geq 1$  for all $t>0$.
\end{theorem}
\begin{proof}
For $Q=Q^2\notin S_0(\mathcal{M}, \tau)$ we have $\mu (t; Q) \geq 1$  for all $t>0$, see \cite[Lemma 3.8]{Bik15}.
Let  $Q=Q^2\in S_0(\mathcal{M}, \tau)$ and $P$ be ``the range'' projection of the idempotent $Q$, see the proof
of  Lemma \ref{l3.4}. Since $QP=P$ and $P\in \cP (\cM )\bigcap \cF (\mathcal{M}, \tau)$, by 
 Lemma \ref{l2.1} we have
$$
1 =\mu (s+t; P)=\chi_{(0, \tau (P)]} (s+t)=\mu (s+t; QP) \leq  \mu (s; P) \mu (t; Q)= \mu (t; Q) 
$$ 
 for all $s, t>0$ with $s+t\leq \tau (P)$. By tending $s$ to $0+$, we obtain
 $\mu (t; Q) \geq 1$  for all $0<t<\tau (P)$. By the right continuity of the function $\mu (t; \cdot )$ we have
$\mu (\tau (P); Q) \geq 1$. 
If $t> \tau (P)$ then $\mu (t; P)=0$; by the equality $PQ=Q$ and by  
 Lemma \ref{l2.1} we obtain
$$
0\leq \mu (t; Q) =\mu (t; PQ) \leq \mu (t-\varepsilon ; P) \mu (\varepsilon ; Q) =0
$$
for all $\varepsilon >0$ with $t-\varepsilon >\tau ( P)$. 

Let $Q\in S(\mathcal{M}, \tau)$ be such that $Q^2=Q$. For the symmetry $U=2Q-I$ we have $U^2=I$ and
 by Lemma \ref{l2.1} obtain
$$
1=\mu (2t; I)=\mu (2t; U^2) \leq \mu (t; U)  \mu (t; U) = \mu (t; U) ^2
$$
for all $t>0$.
\end{proof}

Note that for  $Q\in \mathcal{M}$ such that $Q^2=Q$ the relation 
 $\mu (t; Q)\in \{0\}\bigcup [1, \|Q \|]$ for all $t>0$
was obtained by another way in  \cite[item 1) of Lemma 3.8]{Bik06}. Theorem \ref{t3.7} gives the positive answer to the question by Daniyar Mushtari of year 2010.

\vskip 2mm

{\bf Acknowledgements.}  The work performed under the development program of Volga Region Mathematical Center (agreement no. 075-02-2022-882).

\end{document}